\documentclass[12pt]{amsart}
\usepackage{graphicx}
\usepackage[T1]{fontenc}
\usepackage{amssymb,amsmath,amsfonts, amsthm}
\usepackage{tikz}
\usetikzlibrary[calc,patterns]

\newcommand{\Spin}{\mathrm{Spin}}
\newcommand{\Ind}{\mathrm{Ind}}

\newtheorem{theorem}{Theorem}[section]

\newtheorem{lemma}[theorem]{Lemma}

\newtheorem*{thmNN}{Theorem}

\theoremstyle{definition}

\theoremstyle{remark}

\title{Relative systoles in hyperelliptic translation surfaces}
\author{Corentin Boissy, Slavyana Geninska}

\address{Institut de Math\'ematiques de Toulouse \\ 
Universit\'e Toulouse 3 \\
118 route de Narbonne \\
31062 Toulouse, France
}
\email{geninska@math.univ-toulouse.fr}
\email{corentin.boissy@math.univ-toulouse.fr}

\subjclass[2010]{Primary: 32G15. Secondary: 30F30}
\keywords{Translation surfaces, systoles, moduli spaces}

\begin{document}
\begin{abstract}In this paper we prove that the systole fonction on a connected component of area one translation surfaces admits a local maximum that is not a global maximum if and only if the connected component is not hyperelliptic.

\end{abstract}

\maketitle
\section{Introduction}
This paper deals with flat metric defined by Abelian differentials on compact Riemann surfaces (\emph{translation surfaces}). 
For a translation surface, we define the  \textit{relative systole} $\mathrm{Sys}(S)$ to be the length of the shortest saddle connection of $S$. 
A sequence of area one translation surfaces $(S_n)_{n\in \mathbb{N}}$ in a stratum of the moduli space of translation surfaces leaves any compact set if and only if $\mathrm{Sys}(S_n)\to 0$.
The set of translation surfaces with short relative systole and compactification issues of strata are related to dynamics and counting problems on translation surfaces and have been widely studied in the last 30 years (see for instance \cite{KMS,EMZ,EKZ}).

Here, we are interested in the opposite problem: we study surfaces that are ``far'' from the boundary.  In \cite{BG}, we have characterized global maxima for $Sys$ and we have shown that each stratum of genus greater than or equal to 3 contains local but non global maxima for the function $\mathrm{Sys}$. The constructed surfaces in \cite{BG} are not in the hyperelliptic connected components. %Therefore it is a natural question whether such kind of surfaces exist in these components.

In this paper, we prove that there are no such local maxima in hyperelliptic connected components (Theorem~\ref{th:cc:hyp} in the text), while they exist in every other connected component (Theorem~\ref{th:cc:spin} in the text). This gives us the following characterization.
%This gives a characterization of hyperelliptic connected components.

\begin{thmNN}[Main Theorem]
Let $\mathcal{C}$ be a connected component of a stratum of area one surfaces. The relative systole fonction on $\mathcal{C}$ admits a local maximum that is not a global maximum if and only if $\mathcal{C}$ is not hyperelliptic.
\end{thmNN}

Note that our notion of relative systole is different from the ``true systole''   \emph{i.e.} shortest closed curve that has been studied by Judge and Parlier  in \cite{JP}. In the rest of the paper, for simplicity, if not mentioned otherwise, the term ``systole'' will mean ``relative systole''.

\section{Background}
\subsection{Translation surfaces}
A \emph{translation surface} is a (real, compact, connected) genus $g$ surface $S $ with a translation atlas \emph{i.e.} a triple $(S,\mathcal U,\Sigma)$ such that $\Sigma$ (whose elements are called {\em singularities}) is a finite subset of $S$  and $\mathcal U = \{(U_i,z_i)\}$ is an atlas of $S \setminus \Sigma$ whose transition maps are translations of $\mathbb{C}\simeq \mathbb{R}^2$. We will require that  for each $s\in \Sigma$, there is a neighborhood of $s$ isometric to a Euclidean cone whose total angle is a multiple of $2\pi$.  One can show that the holomorphic structure on $S\setminus \Sigma$ extends to $S$ and that the holomorphic 1-form $\omega=dz_i$ extends to a holomorphic $1-$form on $S$ where  $\Sigma$ corresponds to the zeroes of $\omega$ and maybe some marked points. We usually call $\omega$ an \emph{Abelian differential}.  A zero of $\omega$ of order $k$ corresponds to a singularity of angle $(k+1)2\pi$. By a slight abuse of notation, we authorize the order of a zero to be 0, in this case it corresponds to a regular marked point.

A \emph{saddle connection} is a geodesic segment joining two singularities (possibly the same) and with no singularity in its interior. Integrating $\omega$ along the saddle connection we get a complex number. Considered as a planar vector, this complex number represents the affine holonomy vector of the saddle connection. In particular, its Euclidean length is the modulus of its holonomy vector. 

For $g \geq 1$, we define the moduli space of Abelian differentials $\mathcal{H}_g$ as the moduli space of pairs $(X,\omega)$ where $X$ is a genus $g$ (compact, connected) Riemann surface and $\omega$ non-zero holomorphic $1-$form defined on $X$. The term moduli space means that we identify the points $(X,\omega)$ and $(X',\omega')$ if there exists an analytic isomorphism $f:X \rightarrow X'$ such that 
$f^*\omega'=\omega$.
%The group $\SL(2,\mathbb R)$ naturally acts on the moduli space of translation surfaces by post composition on the charts defining the translation structures.

One can also see a translation surface obtained from a polygon (or a finite union of polygons) whose sides come by pairs, and for each pair, the corresponding segments are parallel and of the same length. These parallel sides are glued together by translation and we assume that this identification preserves the natural orientation of the polygons. In this context, two translation surfaces are identified in the moduli space of Abelian differentials if and only if the corresponding polygons can be obtained from each other by cutting and gluing and preserving the identifications. %Also, the $\SL(2,\mathbb{R})$ action in this representation is just the natural linear action on the polygons. 	

The moduli space of Abelian differentials is stratified by the  combinatorics of the zeroes; we will denote by $\mathcal{H}(k_1,\ldots ,k_r)$ the stratum of $\mathcal{H}_g$ consisting of (classes of) pairs $(X,\omega)$ such that $\omega$ has exactly $r$ zeroes, of order $k_1,\dots ,k_r$ respectively. It is well known that  this space is  (Hausdorff) complex analytic. We have the classical Gauss--Bonnet formula $\sum_i k_i=2g-2$, where $g$ is the genus of the underlying surfaces. We often restrict to the subset $\mathcal{H}_1(k_1,\dots ,k_r)$ of \emph{area one} surfaces.
 Local coordinates for a stratum of Abelian differentials are obtained by integrating the holomorphic 1--form along a basis of the relative homology $H_1(S,\Sigma;\mathbb{Z})$, where $\Sigma$ denotes the set of conical singularities of $S$. 

\subsection{Connected component of strata}\label{background:cc}
Here, we recall the Kontsevich--Zorich classification the connected components of the strata of Abelian differentials \cite{KoZo}.

A translation surface $(X,\omega)$ is \emph{hyperelliptic} if the underlying Riemann surface is hyperelliptic, \emph{i.e.} there is an involution $\tau$ such that $X/\tau$ is the Riemann sphere. In this case $\omega$ satisfies $\tau^*\omega=-\omega$. A connected component of a stratum is said to be \emph{hyperelliptic} if it consists only of hyperelliptic translation surfaces (note that a connected component which is not hyperelliptic might contain some hyperelliptic translation surfaces).

Let $\gamma$ be a simple closed smooth curve parametrized by the arc length on a translation surface that avoids the singularities. Then $t\to \gamma'(t)$ defines a map from $\mathbb{S}^1$ to $\mathbb{S}^1$. We denote by $Ind(\gamma)$ the index of this map. Assume that the translation surface $S$ has only even degree singularities $S\in \mathcal{H}(2k_1,\dots ,2k_r)$. Let $(a_i,b_i)_{i\in \{1,\dots ,g\}}$ be a collection of simple closed curves as above and representing a symplectic basis of the homology of $S$. Then
$$\sum_{i=1}^g (ind(a_i)+1)(ind(b_i)+1) \mod 2$$
is an invariant of connected component ans is called the \emph{parity of the spin structure} (see \cite{KoZo} for details).

Here is a reformulation of the classification of connected component of strata by Kontsevich--Zorich (Theorem~1 and Theorem~2 of \cite{KoZo}).
\begin{theorem}[Kontsevich--Zorich]
Let $\mathcal{H}=\mathcal{H}(k_1,\dots ,k_r)$ be a stratum of genus $g\geq 2$ translation surfaces. 
\begin{itemize}
\item The stratum $\mathcal{H}$ contains a hyperelliptic connected component if and only if $\mathcal{H}=\mathcal{H}(2g-2)$ or $\mathcal{H}=\mathcal{H}(g-1,g-1)$. In this case there is only one hyperelliptic component. In genus two, any stratum is connected (and hyperelliptic).
\item If there exists $i$ such that $k_i$ is odd, or if $g=3$,  then there exists a unique nonhyperelliptic connected component.
\item If $g\geq 4$ and, for all $i$, $k_i$ is even, then there are exactly two nonhyperelliptic connected components distinguished by the parity of the spin structure.
\end{itemize}
\end{theorem}

The following lemma is classical and will be useful in the next section.
\begin{lemma}\label{lemme:hyp}
Let $S$ be a translation surface in a hyperelliptic connected component and let $\gamma$ be a saddle connection. Then $\gamma$ and $\tau(\gamma)$ are homologous.
\end{lemma}

\section{Hyperelliptic connected component}
In this section, we prove the first part of the Main Theorem.

\begin{theorem}\label{th:cc:hyp}
Let $\mathcal{C}$ be a hyperelliptic connected component of the moduli space of Abelian differentials. 
Let $S\in \mathcal{C}$ be a local maximum of the relative systole function $Sys$. Then $S$ is a global maximum for $Sys$ in $\mathcal{C}$.
\end{theorem}

The proof uses the following technical lemma. We postpone its proof to the end of the section.
\begin{lemma}\label{lem:decrease:area:disk}
Let $D$ be a translation surface that is topologically a disk and  whose boundary consists of $n$-saddle connections  (a ``$n$-gon'') with $n\geq 4$. We assume that all boundary saddle connections are of length greater than or equal to 1. Then, we can continuously deform $D$ so that its area decreases and the boundary saddle connections of length 1 remain of length 1.
\end{lemma}

\begin{proof} [Proof of Theorem~\ref{th:cc:hyp}]
Let $S\in\mathcal{C}$ be a translation surface that such that $Sys(S)$ is not a global maximum. We use the same normalization as in \cite{BG}: after rescaling the surface we assume that $Sys(S)$ equals 1, and we will continuously deform $S$ so that $Sys(S)$ remains 1 and $\mathrm{Area}(S)$ decreases. 

Let $\gamma_1,\dots ,\gamma_r$ be the set of saddle connections realizing the systole.  
Recall that $\gamma_1,\dots ,\gamma_r$ are sides of the Delaunay triangulation and 
that global maxima correspond to surfaces which Delaunay cells are only equilateral triangles (see Lemma~3.1 and Theorem~3.3 in \cite{BG}). 
Let $C_1,\dots ,C_k$ be the connected components of $S\backslash \cup_i \gamma_i$. Up to renumbering we can assume that $C_1$ is not a triangle.  We consider $\tau(C_1)$, where $\tau$ is the hyperelliptic involution. We study the two possible cases whether $\tau(C_1)$ equals $C_1$ or not. Note that $C_1$ does not contain any singularity in its interior since there are at most two singularities in $S$ and if there are two singularities $P_1,P_2$ we must have $\tau(P_1)=P_2$.
 
{\bf Case 1.} We first assume that $\tau(C_1)\neq C_1$. Since the hyperelliptic involution preserves $\cup_i \gamma_i$, we have (up to renumbering) $\tau(C_1)=C_2$. 

We observe that $C_1$  has only one boundary component. Indeed, suppose that there are more than one such components and consider a saddle connection $\eta$ in $C_1$ that joins a singularity of one boundary component to a singularity of another boundary component. Then $\tau(\eta)$ is a curve in $C_2$ and must be homologous to $\eta$ by Lemma~\ref{lemme:hyp}. But $C_1\backslash \eta$ is connected, and hence $S\backslash (\eta\cup \tau(\eta))$ is connected, which is a contradiction. Therefore $C_1$ is a disk because it embeds in $S/\tau$ which is a sphere. 

Since the boundary of $C_1$ consists of at least 4 saddle connections of length 1,  by Lemma~\ref{lem:decrease:area:disk}, we can continuously decrease its area while keeping the boundary saddle connections of length 1.

This continuous deformation of $C_1$ leads to the following area decreasing continuous deformation of $S$: 
%\begin{itemize}
%\item 

The component $C_2$ is deformed in a symmetric way as $C_1$.

%Now, for each saddle connection $\gamma$ in the boundary of $C_1$. If  $\gamma=\tau(\gamma)$, there is nothing to do.
%
%Otherwise, $\gamma$ and $\tau(\gamma)$ are homologous, and therefore cut $S$ into two connected components that correspond to the two components of the complementary of $[\gamma]$ in the quotient sphere $S/\tau$. Since $[C_1]=[C_2]$, then $C_1,C_2$ are in the same connected component of $S\backslash (\gamma\cup \tau(\gamma))$. We denote by $D_\gamma$ the other component. We observe that if $\gamma_1,\gamma_2$ are two distinct saddle connections in the boundary of $C_1$ satisfying $\tau(\gamma_i)\neq \gamma_i$, then $D_{\gamma_1}$ and $D_{\gamma_2}$ are disjoint. In particular, we can rotate all such components $D_{\gamma}$ independently, in a compatible way with the deformation of $C_1,C_2$.

For each saddle connection $\gamma$ in the boundary of $C_1$, the components of $S\backslash (\gamma\cup \tau(\gamma))$ correspond to components of the complementary of $[\gamma]$ in the quotient sphere $S/\tau$. Since $[C_1]=[C_2]$, then $C_1,C_2$ are in the same connected component of $S\backslash (\gamma\cup \tau(\gamma))$. We denote by $D_\gamma$ the other component. Note that $D_{\gamma}$ is empty if $\gamma=\tau(\gamma)$. We observe that if $\gamma_1,\gamma_2$ are two distinct saddle connections in the boundary of $C_1$, then $D_{\gamma_1}$ and $D_{\gamma_2}$ are disjoint.  In particular, we can rotate all such components $D_{\gamma}$ independently in a compatible way with the deformation of $C_1,C_2$ and we glue these components in the natural way. The area of each $C_i$ decreases while the area of the $D_\gamma$ remains constant. Therefore the total area of the surface decreases.

\medskip
{\bf Case 2.} Now we assume that $\tau(C_1)=C_1$. 

We claim that we can cut $C_1$ along saddle connections and obtain two discs $A$ and $B$ such that $\tau(A)=B$ and for each saddle connection $\gamma$ in the boundary of $A$, either $\gamma$ is of length 1, or $\tau(\gamma)=\gamma$ (equivalently, $\gamma$ is also a boundary saddle connection of $B$). 

To prove the claim, we first consider the Delaunay cells of $S$. Recall that the shortest geodesics (hence the boundary saddle connections of $C_1$) are sides of the Delaunay cells (see \cite{BG}, Lemma~3.1). This induces a decomposition of $C_1$ into Delaunay cells, and this decomposition is preserved by the involution $\tau$ because of the uniqueness of the Delaunay cell decomposition. We define a Delaunay subdivision $\mathcal{D}$ in the following way: for each Delaunay cell $d$, if $\tau(d)\neq d$, then $d,\tau(d) \in \mathcal{D}$. If $\tau(d)=d$ and since $d$ is cyclic it can be cut by a diagonal into two polygons $d'$ and $d''=\tau(d')$, then $d',d''\in \mathcal{D}$.

Now we use the following algorithm:
\begin{itemize}
\item We start from a pair $d_0,\tau(d_0)$ in $\mathcal{D}$ and let $A_0=d_0$ and $B_0=\tau(d_0)$.
\item Suppose we have constructed the disks $A_k$ and $B_k$ such that $\tau(A_k)=B_k$ and $A_k, B_k$ are union of elements in $\mathcal{D}$. 

If $A_k\cup B_k\neq C_1$, there exists an element $d_{k+1}\in \mathcal{D}$ adjacent to $A_k$ along a saddle connection $\gamma_k$ (and $\tau(d_{k+1})\in \mathcal{D}$ is adjacent to $B_k$ along $\tau(\gamma_k)$). We define $A_{k+1}$ by gluing $A_k$ and $d_{k+1}$ along $\gamma_k$. Note that $\gamma_k$ is the only saddle connection in the common boundary of $A_{k}$ and $d_{k+1}$ because otherwise $S\backslash (\gamma_k\cup \tau(\gamma_k))$ would be connected which is impossible in the hyperelliptic connected component.

If $A_k\cup B_k=C_1$, we define $A=A_k$, $B=B_k$.
\end{itemize}

%We claim that we can cut $C_1$ along saddle connections and obtain a single disk which is preserved by the involution $\tau$. Indeed, we consider the Delaunay cells of $S$. Recall that the shortest geodesics (hence the boundary saddle connections of $C_1$) are sides of the Delaunay cells (see \cite{BG}, Lemma~3.1). This induces a decomposition of $C_1$ into Delaunay cells, and this decomposition is preserved by the involution $\tau$ because of the uniqueness of the Delaunay cell decomposition. 
%If the boundary of the disk $A$ consists of three saddle connections, then $A$ is triangle with sides of lengths at least one

The boundary of the disk $A$ consists of $n\geq 3$ saddle connections of lengths at least 1. If $n\geq 4$, then from Lemma~\ref{lem:decrease:area:disk}, it can be continuously deformed so that the area decreases and the boundary saddle connections of length 1 remain of length one. Otherwise $A$ is a triangle but cannot be a equilateral triangle, hence it can also be deformed as above.

 We deform $B$ in a symmetric way. Note that $A$ and $B$ are directly glued together in $C_1$ along the boundary saddle connections of lengths greater than one.
Therefore the possible changes of these saddle connections are not a problem. The deformation of $S\backslash C_1$ is treated  as in the previous case.

% For the boundary saddle connections of $C_1$ of length 1, we proceed in a similar way to the previous case.

%This continuous deformation of $C_1$ leads to the following area decreasing continuous deformation of $S$: the boundary saddle connections  of $C_1$ that are of length greater than 1 are pairwise glued (it is possible because the deformation is symmetric). For the boundary saddle connections of length 1, we proceed in a similar way to the previous case.

 %We consider the connected component $D_\gamma$ of $S\backslash (\gamma\cup \tau(\gamma))$ that does not contain $C_1$ and $C_2$
%\end{itemize}

\end{proof}

\begin{proof}[Proof of Lemma~\ref{lem:decrease:area:disk}]
The sum of the boundary angles (coming from the intersection of two consecutive boundary saddle connections) of $D$ equals $(n-2)\pi$. Therefore $D$ has boundary angles smaller than $\pi$. If such a boundary angle has a corresponding boundary saddle connection which is of length greater than 1, then by slightly changing its length  we can decrease the area of the corresponding triangle and hence of $D$.

So we can assume that for each boundary angle smaller than $\pi$, the two adjacent saddle connections are of length 1.
We claim that we can find two consecutive angles such that one is smaller than $\pi$ and the other is smaller than $2\pi$ (note that since $D$ is not necessarily embedded in the plane, it can have boundary angles greater than $2\pi$). Indeed, consider the sequence of consecutive boundary angles of $D$. If each time an angle is smaller than $\pi$, the following one is greater than or equal to $2\pi$, then the global sum will be greater than $n\pi$, which is not possible.

Now we consider the 3 consecutive saddle connections corresponding to these two angles, and see them as a broken line on the plane. We close this line by adding a segment $t$ to obtain a quadrilateral $\mathcal{Q}$ (that can be also crossed). Without loss of generality, we can assume that $t$ is horizontal. We have
$$\mathrm{Area}(D)=\mathrm{Area}(D_0)+\mathrm{Area}_{alg}(\mathcal{Q})$$
where $D_0$ is the translation surface obtained by ``replacing'' the broken line by $t$ (see Figure~\ref{fig:D:et:Q}). Here $\mathrm{Area}_{alg}(\mathcal{Q})$ means that the part of $\mathcal{Q}$ below the segment $t$ is counted negatively. 

{\bf Claim:} We can deform continuously $\mathcal{Q}$ without changing the lengths of its sides so that $\mathrm{Area}_{alg}(\mathcal{Q})$ decreases.

Denote by $MNPQ$ the quadrilateral $\mathcal{Q}$ and by $a,b,c,d$ the lengths of the sides of $\mathcal{Q}$ with $a$ being the length of the segment $t=MN$. Denote by $\alpha$ the oriented angle from $MN$ to $MQ$, and by $\gamma$ its opposite angle in $\mathcal{Q}$ (\emph{i.e.} the angle from $PQ$ to $PN$). Without loss of generality we assume that $b=c=1$, $d\geq 1$ and $0<\gamma<\pi$ (in fact we must have $\gamma>\pi/3$ otherwise there would be a smallest saddle connection). We also have $-\pi<\alpha<\pi$. Also, the sides $NP$ and $QM$ do not intersect since it would imply intersecting boundary saddle connections in $D$ (see Figure~\ref{fig:D:et:Q}).

%%% TIKZ TIKZ TIKZ TIKZ TIKZ TIKZ TIKZ TIKZ TIKZ TIKZ TIKZ TIKZ
%%%%%%%%% S in H(-2,...)
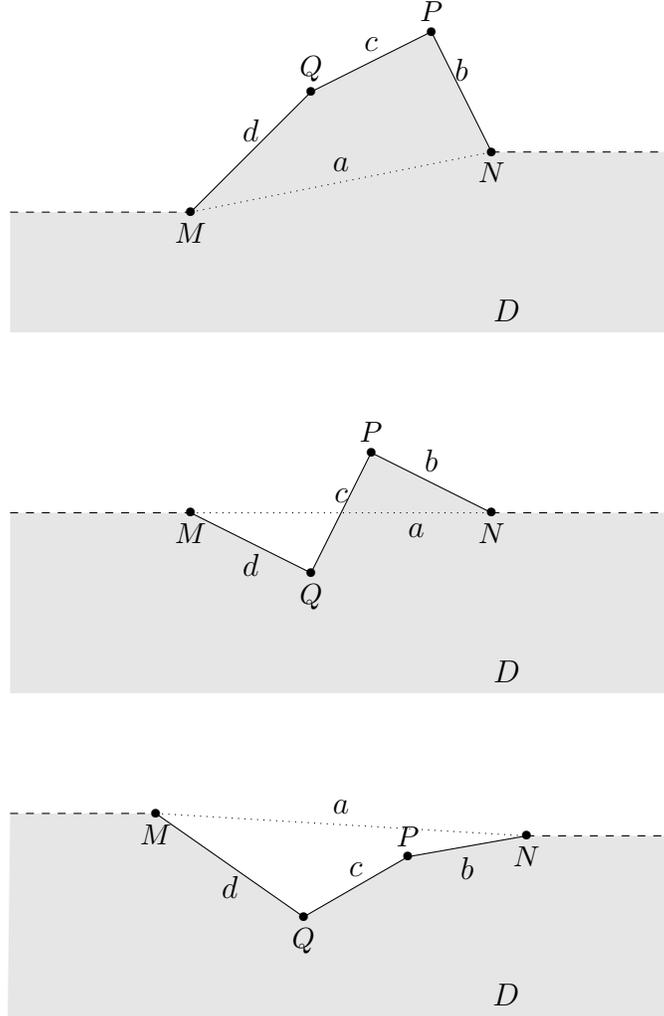
\begin{figure}[htb]
\begin{tikzpicture}[scale=0.8]
\coordinate (z1) at (2,2);
\coordinate (z2) at (2,1);
\coordinate (z3) at (1,-2) ;

%DESSIN 1
\fill [fill=gray!20] (-3,0)--++(3,0) --++ (z1) --++ (z2)  --++ (z3) --++ (3,0) --++(0,-3) --++(-11,0) node[near start, above] {$D$}--cycle;

\draw (0,0)  node {\tiny $\bullet$} node[below] {\small $M$}  --++ (z1) node[midway,above] {$d$} node {\tiny $\bullet$} node[above] {\small $Q$} --++ (z2) node[midway,above] {$c$} node {\tiny $\bullet$} node[above] {\small $P$}--++ (z3) node[midway,above] {$b$} node {\tiny $\bullet$} node[below] {\small $N$};

\draw [dashed] (-3,0)--++(3,0);
\draw [dashed] ($(z1)+(z2)+(z3)$)--++(3,0);
\draw [dotted] (0,0)  --($(z1)+(z2)+(z3)$) node[midway,above] {$a$};                   

%DESSIN 2
\coordinate (z1) at (2,-1);
\coordinate (z2) at (1,2);
\coordinate (z3) at (2,-1) ;

\fill [fill=gray!20] (-3,-5)--++(3,0) --++ (z1) --++ (z2)  --++ (z3) --++ (3,0) --++(0,-3) --++(-11,0) node[near start, above] {$D$}--cycle;

\draw (0,-5)  node {\tiny $\bullet$} node[below] {\small $M$}  --++ (z1) node[midway,below] {$d$} node {\tiny $\bullet$} node[below] {\small $Q$}
                     --++ (z2) node[midway,above] {$c$} node {\tiny $\bullet$} node[above] {\small $P$}
                     --++ (z3) node[midway ,above] {$b$} node {\tiny $\bullet$} node[below] {\small $N$};

\draw [dashed] (-3,-5)--++(3,0);
\draw [dashed] ($(0,-5)+(z1)+(z2)+(z3)$)--++(3,0);
\draw [dotted] (0,-5)  --++($(z1)+(z2)+(z3)$) node[near end,below] {$a$};                    

%DESSIN 3
\coordinate (z1) at (-35:3);
\coordinate (z2) at (30:2);
\coordinate (z3) at (10:2) ;

\fill [fill=gray!20] (-3,-10)--++(2.42,0) --++ (z1) --++ (z2)  --++ (z3) --++ (2.42,0) --++(0,-3) --++(-11.04,0) node[near start, above] {$D$}--cycle;

\draw (-0.58,-10)  node {\tiny $\bullet$} node[below] {\small $M$}  --++ (z1) node[midway,below] {$d$} node {\tiny $\bullet$} node[below] {\small $Q$}
                     --++ (z2) node[midway,above] {$c$} node {\tiny $\bullet$} node[above] {\small $P$}
                     --++ (z3) node[midway ,below] {$b$} node {\tiny $\bullet$} node[below] {\small $N$};

\draw [dashed] (-3,-10)--++(2.42,0);
\draw [dashed] ($(-0.58,-10)+(z1)+(z2)+(z3)$)--++(2.42,0);
\draw [dotted] (-0.58,-10)  --++($(z1)+(z2)+(z3)$) node[midway,above] {$a$};    

\end{tikzpicture}
\caption{The disk $D$ and the quadrilateral $\mathcal{Q}$ in 3 configurations.}
\label{fig:D:et:Q}
\end{figure}

Denote  $K=\mathrm{Area}_{alg}(\mathcal{Q})$. We compute $K$ by adding the (algebraic) area of the triangles $MNQ$ and $NPQ$. We obtain
\begin{equation}
K=\frac{1}{2}\left(ad\sin(\alpha)+bc\sin(\gamma)\right).
\end{equation}
The expression of the length of $NQ$ gives the second equality:
\begin{equation}
a^2+d^2-2ad\cos(\alpha)=b^2+c^2-2bc\cos(\gamma).
\end{equation}
These two equations imply the Bretschneider's formula for $\mathcal{Q}$:
\begin{equation}
K^2=(s-a)(s-b)(s-c)(s-d)-abcd\cos(\frac{\alpha+\gamma}{2}),
\end{equation}
where $s=\frac{a+b+c+d}{2}$.

From now on, we fix $a,b,c,d$ and study the variations of the area with respect to $\alpha,\gamma$. Equation (2) implies that $\gamma$ depends differentially on 
$\alpha$. Hence we can write $K=K(\alpha)$. We need to prove that $K'(\alpha)\neq 0$ or $K(\alpha)$ is a strict local maximum (note that $\alpha$ varies in an open set). We have:
$$(K^2)'(\alpha)=abcd(1+\gamma'(\alpha))\sin(\frac{\alpha+\gamma}{2})\cos(\frac{\alpha+\gamma}{2}).$$
We assume that $K'(\alpha)=0$, hence $(K^2)'(\alpha)=0$, hence we are in one of the following three cases:
\begin{enumerate}
\item $\sin(\frac{\alpha+\gamma}{2})=0$. The conditions $-\pi<\alpha<\pi$ and $0<\gamma<\pi$ imply $\alpha=-\gamma<0$.  Hence the quadrilateral $\mathcal{Q}$ have self-intersections. Since the sides $NP$ and $QM$ do not intersect, the sides $MN$ and $PQ$ intersect. 
The condition $\alpha=-\gamma$ implies that the points $M,N,P,Q$ are cocyclic, and since $b=c=1$ we must have $d< 1$ which is a contradiction.
\item $\cos(\frac{\alpha+\gamma}{2})=0$. Then $\alpha+\gamma=\pi$, and therefore $\alpha>0$ and hence $K>0$. From $(2)$ and $(3)$, we have a strict local maximum for $K^2$ and therefore for $K$.
\item $\gamma'(\alpha)=-1$. By differentiating $(2)$ and using $(1)$, we see that $K=0$, hence $\mathcal{Q}$ has a self-intersection $I=MN\cap PQ$. By differentiating $(1)$ and using (2), we obtain $K'(\alpha)=0=\frac{1}{2}(\frac{a^2+d^2}{2}-1)$, hence $a^2+d^2=2$. Since $d\geq 1$, we have $a\leq 1 \leq d$. However, triangle inequalities for $INP$ and $IMQ$ give $a+c>d+b$ and hence $a>d$, which is a contradiction.
\end{enumerate}

\end{proof}

\section{Nonhyperelliptic connected components}
In this section, we prove the second part of the Main Theorem.

\begin{theorem}\label{th:cc:spin}
Each nonhyperelliptic connected component of each stratum of area one surfaces with no marked points contains local maxima of the function $\mathrm{Sys}$ that are not global.
\end{theorem}
The proof is a refinement of the proof of Theorem~4.7 in \cite{BG}.

We will need the following lemma, which is  a refinement of Lemma~3.2 (2) in \cite{BG}.
\begin{lemma}\label{lemme:sc:indice}
Let $\mathcal{C}\subset \mathcal{H}(2k_1,\dots ,2k_r)$ be a connected component of a stratum of abelian differentials with $2k_1,\dots ,2k_r\geq 0$. There exists a surface $S\in \mathcal{C}$ realizing the  global maximum for the systole function, and such that there exists a shortest saddle connection $\gamma$ joining a singularity of degree $2k_1$ to itself and $\mathrm{Ind}([\gamma])=0$.
\end{lemma}
\begin{proof}
We do as in the proof of Lemma~3.2 in \cite{BG}. There exists a square tiled surface in $\mathcal{C}$ with singularities on each corner of the squares as in Figure~\ref{squares:to:triangles}, and we can assume that the top left horizontal segment identifies with the bottom left horizontal segment (see Figure~\ref{squares:to:triangles}). After a suitable transformation as in the figure, we obtain the required surface.

\begin{figure}[htb]
\begin{tikzpicture}[scale=0.8]
\coordinate (v0) at (1,0);
\coordinate (v1) at (0.5,0.866);
\coordinate (v2) at ($(v0)-(v1)$);
\coordinate (u) at (0,1);
\coordinate (v3) at ($-1*(v0)$);
\coordinate (v4) at ($-1*(v1)$);

\draw (0,0) --++ (v0) node[midway,below] {\small $1$} node{$\bullet$} --++ (v0) node{$\bullet$} --++ (v0) node{$\bullet$} --++ (v0) node{$\bullet$} --++ (v0) node{$\bullet$} --++ (v0) node{$\bullet$} --++ (v0) node{$\bullet$} --++ (v0) node{$\bullet$} --++ (u) node{$\bullet$} --++ (v3) node{$\bullet$} --++ (v3) node{$\bullet$}--++ (v3) node{$\bullet$}--++ (v3) node{$\bullet$}--++ (v3) node{$\bullet$}--++ (v3) node{$\bullet$}--++ (v3) node{$\bullet$}--++ (v3) node[midway,above] {\small $1$} node{$\bullet$} --++ (0,-1) node{$\bullet$};

\draw (0,-3) --++ (v0) node[midway,below]{\small $1$} node{$\bullet$} --++ (v0) node{$\bullet$} --++ (v0) node{$\bullet$} --++ (v0) node{$\bullet$} --++ (v0) node{$\bullet$} --++ (v0) node{$\bullet$} --++ (v0) node{$\bullet$} --++ (v0) node{$\bullet$} --++ (v1) node{$\bullet$} --++ (v3) node{$\bullet$} --++ (v3) node{$\bullet$}--++ (v3) node{$\bullet$}--++ (v3) node{$\bullet$}--++ (v3) node{$\bullet$}--++ (v3) node{$\bullet$}--++ (v3) node{$\bullet$}--++ (v3) node[midway,above] {\small $1$} node{$\bullet$} --++ (v4) node[midway, left] {$\gamma$} node{$\bullet$};

\draw ($(0,-3)+(v1)$)--++(v2)--++(v1)--++(v2)--++(v1)--++(v2)--++(v1)--++(v2)--++(v1)--++(v2)--++(v1)--++(v2)--++(v1)--++(v2)--++(v1)--++(v2);

\draw[very thick,->] (4,-0.5)--++(0,-1);
\end{tikzpicture}
\caption{A global maximum with a closed shortest saddle connection $\gamma$ sastisfying $\mathrm{Ind}([\gamma])=0$}
\label{squares:to:triangles}
\end{figure}
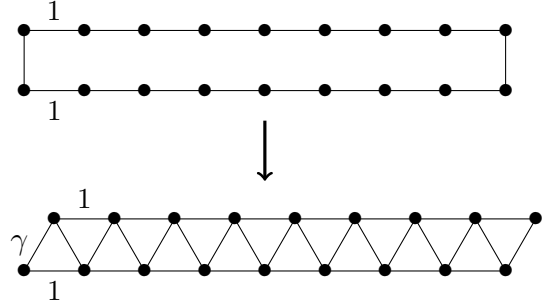

\end{proof}

\begin{proof}[Proof of Theorem \ref{th:cc:spin}]
In Theorem~4.7 in \cite{BG} we have already constructed examples in each genus $g\geq 3$ stratum. By Theorem~\ref{th:cc:hyp} each such example is in a nonhyperelliptic component. So it remains to construct new examples only in strata with more than one nonhyperelliptic connected component.

From the Theorem of Kontsevich--Zorich stated in Section~\ref{background:cc}, there is more than one nonhyperelliptic connected component only for genus $g\geq 4$ strata with only even degree singularities and in this case there are two nonhyperelliptic components distinguished by the parity of the spin structure.
%In this case, the nonhyperelliptic connected components are classified by the parity of the spin structure, which is given by the formula:
%$$\Spin(S)=\sum_{i=1}^{g} (\mathrm{Ind}(a_i)+1)(\mathrm{Ind}(b_i)+1) \mod 2,$$
%where $(a_i,b_i)_{i\in \{1,\dots ,g\}}$ is a symplectic basis of $H_1(S,\mathbb{Z})$.

In Figure~\ref{fig:examples} we give surfaces $S_{2,0}\in \mathcal{H}(2,0)$ and $S_{2,0,0}\in \mathcal{H}(2,0,0)$ that are local but nonglobal maxima for the systole function. 
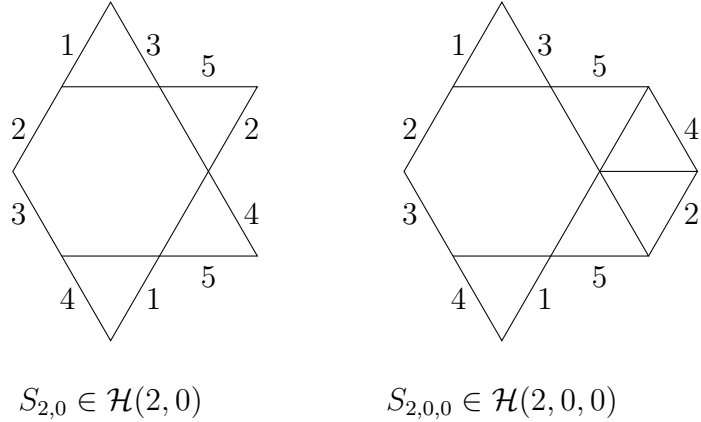
\begin{figure}[htb]
\begin{tikzpicture}[scale=1.3]
\coordinate (v0) at (1,0);
\coordinate (v1) at (0.5,0.866);
\coordinate (A) at (0,0);
\coordinate (B) at (4,0);

\draw  (A) --++ (v0) --++ (v1) --++ ($(v1)-(v0)$) --++ ($-1*(v0)$)--++ ($-1*(v1)$) node[midway,left] {$2$}--++ ($(v0)-(v1)$) node[midway,left] {$3$};
\draw (A) --++ ($(v0)-(v1)$)  node[midway,left] {$4$}  --++ (v1) node[midway,right] {$1$} --++ (v0) node[midway,below] {$5$} --++ ($(v1)-(v0)$) node[midway,right] {$4$}--++ (v1) node[midway,right] {$2$} --++ ($-1*(v0)$) node[midway,above] {$5$} --++ ($(v1)-(v0)$) node[midway,right] {$3$} --++ ($-1*(v1)$) node[midway,left] {$1$};
\draw ($(A)+(0.5,-1.5)$) node {$S_{2,0}\in \mathcal{H}(2,0)$};

\draw  (B) --++ (v0) --++ (v1) --++ ($(v1)-(v0)$) --++ ($-1*(v0)$)--++ ($-1*(v1)$) node[midway,left] {$2$}--++ ($(v0)-(v1)$) node[midway,left] {$3$};
\draw (B) --++ ($(v0)-(v1)$)  node[midway,left] {$4$}  --++ (v1) node[midway,right] {$1$} --++ (v0) node[midway,below] {$5$} --++ ($(v1)-(v0)$) --++ (v1) --++ ($-1*(v0)$) node[midway,above] {$5$} --++ ($(v1)-(v0)$) node[midway,right] {$3$} --++ ($-1*(v1)$) node[midway,left] {$1$};
\draw ($(B)+(v0)+(v1)$) --++ (v0) --++ ($(v1)-(v0)$)  node[midway,right] {$4$};
\draw ($(B)+(v0)+(v0)$) --++ (v1)  node[midway,right] {$2$};

\draw ($(B)+(0.5,-1.5)$) node {$S_{2,0,0}\in \mathcal{H}(2,0,0)$};
\end{tikzpicture}

\caption{Local but nonglobal maxima in $\mathcal{H}(2,0)$ and $\mathcal{H}(2,0,0)$}
\label{fig:examples}

\end{figure}

We consider the following construction: start from the surface $S_{2,0}$ and a surface $M$ that is a global maximum for $Sys$ in $\mathcal{H}(2k_1,\dots ,2k_r)$. There exists a shortest saddle connection $\gamma_1$ in $S_{2,0}$ joining the two singularities. By Lemma~\ref{lemme:sc:indice}, we can assume that there exists a shortest saddle connection $\gamma_2$  in $M$ joining the singularity  of degree $2k_1$ to itself and such that $\Ind([\gamma_2])=0$.  We can assume that $\gamma_1,\gamma_2$ are vertical and of the same length. Now we glue the two surfaces by the following classical surgery: cut the two surfaces along $\gamma_1$ and $\gamma_2$, and glue the left side of $\gamma_1$ with the right side of $\gamma_2$ and the right side of $\gamma_1$ with the left side of $\gamma_2$. We get a surface $S$ in $\mathcal{H}(2k_1+4,2k_2,\dots ,2k_r)$ that satisfies the hypothesis of Theorem~4.1 in \cite{BG} and hence is a local but nonglobal maximum. By Theorem~\ref{th:cc:hyp}, the surface $S$ is necessarily in a nonhyperelliptic component.

We compute $\Spin(S)$: we choose a symplectic basis $(a_i,b_i)_{i}$ of $H_1(M,\mathbb{Z})$ such that $[\gamma_2]=a_1$. Then a simple computation gives
\begin{equation}\label{eq:spin}
\Spin(S)=\Spin(S_{0,2})+\Spin(M)+\Ind(a_1)+1 \mod 2. \end{equation}
Since $\Ind(a_1)=0$, we have
$$\Spin(S)=\Spin(S_{0,2})+\Spin(M)+1 \mod 2. $$
When $\sum_{i}2 k_i\geq 4$, we can prescribe any value of $\Spin(M)$ by choosing $M$ in a suitable component and in this way we can obtain any possible value for $\Spin(S)$. Note that this is also true for $M\in \mathcal{H}(4)$ or $M\in  \mathcal{H}(2,2)$. Indeed, in these strata there are two components,  the hyperelliptic one and the nonhyperelliptic one, and  the spin structure distinguishes them (see \cite{KoZo}, Theorem~2 and Corollary~5).% (which has odd spin structure).

By this construction, we obtain a local but non global maximum for $Sys$ in any (nonhyperelliptic) connected component of any stratum $\mathcal{H}(2n_1,\dots ,2n_r)$ for $r\geq 1$, as soon as $\sum_i 2n_i\geq 8$ and $2n_j\geq 4$ for at least one $j\in \{1,\dots ,r\}$.

We do an analogous construction as above starting from $S_{2,0,0}$ (see Figure~\ref{fig:examples}) and $M\in \mathcal{H}(0,2^r)$ with $\gamma_1\in S_{2,0,0}$ joining the two marked points, and $\gamma_2\in M$ joining the marked point to itself. We obtain a local but nonglobal maximum in $\mathcal{H}(2^{r+2})$. For $r\geq 2$ we can choose the spin structure of $M$ and thus get $S$ in any nonhyperelliptic component of $\mathcal{H}(2^{r+2})$. Note that for $r=1$, we get $S\in \mathcal{H}(2,2,2)$ with odd spin structure.

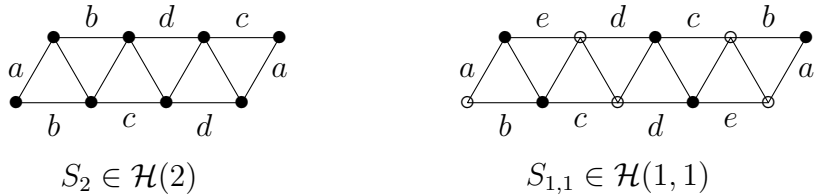
\begin{figure}[htb]
\begin{tikzpicture}[scale=1]
\coordinate (v0) at (1,0);
\coordinate (v1) at (0.5,0.866);
\coordinate (v2) at ($(v0)-(v1)$);
\coordinate (u) at (0,1);
\coordinate (v3) at ($-1*(v0)$);
\coordinate (v4) at ($-1*(v1)$);

\draw (0,0) --++ (v0) node[midway,below]{$b$} node{$\bullet$} --++ (v0) node[midway,below]{$c$} node{$\bullet$} --++ (v0) node[midway,below]{$d$} node{$\bullet$} --++ (v1) node[midway,right]{$a$} node{$\bullet$} --++ (v3) node[midway,above]{$c$} node{$\bullet$} --++ (v3) node[midway,above]{$d$} node{$\bullet$}--++ (v3) node[midway,above]{$b$} node{$\bullet$}--++ (v4) node[midway, left] {$a$} node{$\bullet$};

\draw ($(0,0)+(v1)$)--++(v2)--++(v1)--++(v2)--++(v1)--++(v2);

\draw (6,0) --++ (v0) node[midway,below]{$b$} node{$\bullet$} --++ (v0) node[midway,below]{$c$} node{$\circ$} --++ (v0) node[midway,below]{$d$} node{$\bullet$} --++ (v0) node[midway,below]{$e$} node{$\circ$} --++ (v1) node[midway,right]{$a$} node{$\bullet$} --++ (v3) node[midway,above]{$b$} node{$\circ$} --++ (v3) node[midway,above]{$c$} node{$\bullet$} --++(v3) node[midway,above]{$d$} node{$\circ$}--++ (v3) node[midway,above]{$e$} node{$\bullet$}--++ (v4) node[midway, left] {$a$} node{$\circ$};

\draw ($(6,0)+(v1)$)--++(v2)--++(v1)--++(v2)--++(v1)--++(v2) --++(v1)--++(v2);

\draw (1.5,-1) node {$S_{2}\in \mathcal{H}(2)$};
\draw (8,-1) node {$S_{1,1}\in \mathcal{H}(1,1)$};

%\draw[very thick,->] (4,-0.5)--++(0,-1);
\end{tikzpicture}
\caption{Global maxima in $\mathcal{H}(2)$ and $\mathcal{H}(1,1)$}
\label{fig:h2:h11}
\end{figure}

 There remain the following cases:
\begin{itemize}
%\item The strata $\mathcal{H}(2,2)$, and $\mathcal{H}(4)$: In each of the two strata there is only one nonhyperelliptic component and the example of local but nonglobal maximum for $Sys$ produced in Theorem~4.7 in \cite{BG} is necessarily in this connected component.
\item $\mathcal{H}(6)$. We do the same construction as above starting from $S_{2,0}$ and $M\in \mathcal{H}(2)$. We consider for $M\in \mathcal{H}(2)$ the surface $S_2$ in Figure~\ref{fig:h2:h11}. We see that $[a],[b]$ in  this figure have different indices $\mod 2$. Hence choosing $\gamma_2=a$ or $\gamma_2=b$ gives surfaces with different Spin structure (see Equation (\ref{eq:spin})). 
\item $\mathcal{H}(4,2)$.  We do the same as for $\mathcal{H}(6)$, starting from $S_{2,0,0}$ and $M=S_2$.
\item The even component of $\mathcal{H}(2,2,2)$. We do the same construction but starting from $S_{2,0,0}$ and $M\in \mathcal{H}(1,1)$ the surface $S_{1,1}$ in Figure~\ref{fig:h2:h11}. We consider $\gamma_2=a$ (joining the two singularities of degree 1). By a direct computation, we see that the above construction gives a surface  $S\in \mathcal{H}(2,2,2)$ with $\Spin(S)=0 \mod 2$. 
\end{itemize}

\end{proof}

\end{document}